\documentclass[11pt]{article}
\usepackage[utf8]{inputenc}
\usepackage{mathtools}
\usepackage{amssymb,amsmath,amsthm,graphicx}
\usepackage{enumitem}
\usepackage{authblk}
\pdfoutput=1
\setlist[itemize]{nosep,topsep=3pt,itemsep=3pt}
\setlist[enumerate]{nosep,topsep=3pt,itemsep=3pt}

\usepackage[dvipsnames]{xcolor}
\usepackage{svg}
\usepackage{multirow}
\usepackage[colorlinks=true,linktoc=page]{hyperref}
\usepackage{cleveref}

\usepackage[compact]{titlesec}
\usepackage{titletoc}

\usepackage{lipsum}
\usepackage[tikz]{bclogo}

\usepackage{quoting}
\quotingsetup{vskip=0pt}

\usepackage{times}
\normalfont
\usepackage[T1]{fontenc}

\usepackage[american]{babel}

\usepackage{float}
\usepackage{mathpazo}

\usepackage[varg]{pxfonts} %
\usepackage[full]{textcomp}
\usepackage[scr=rsfso]{mathalfa}%
\usepackage{bm} %

\usepackage{libertine}
\usepackage[T1]{fontenc}

\usepackage{amsthm,amsfonts,framed}
\usepackage[letterpaper, margin=0.85in]{geometry}
\usepackage{physics, graphicx}

\usepackage[compact]{titlesec}
\titlespacing{\section}{0pt}{1.5ex}{0ex}
\titlespacing{\subsection}{0pt}{1.5ex}{0ex}
\titlespacing{\subsubsection}{0pt}{1ex}{0ex}
\titlespacing{\paragraph}{0pt}{1.5ex}{1ex}
\usepackage{slantsc}
\newcommand{\arctanh}{\textnormal{arctanh}}

\hypersetup{colorlinks,
    linkcolor=blue,
    citecolor=ForestGreen,  
    urlcolor=Mahogany,
}

\usepackage{microtype}

\newcommand{\handout}[5]{
   \renewcommand{\thepage}{#1-\arabic{page}}
   \noindent
   \begin{center}
   \framebox{
      \vbox{
    \hbox to 5.78in { {\bf #1}
     	 \hfill #2 }
       \vspace{4mm}
       \hbox to 5.78in { {\Large \hfill #5  \hfill} }
       \vspace{2mm}
       \hbox to 5.78in { {\it #3 \hfill #4} }
      }
   }
   \end{center}
   \vspace*{4mm}
}

\newenvironment{proof-sketch}{\noindent{\bf Sketch of Proof:}\hspace*{1em}}{\qed\bigskip}
\newenvironment{proof-idea}{\noindent{\bf Proof Idea}\hspace*{1em}}{\qed\bigskip}
\newenvironment{proof-of-lemma}[1]{\noindent{\bf Proof of Lemma #1}\hspace*{1em}}{\qed\bigskip}
\newenvironment{proof-attempt}{\noindent{\bf Proof Attempt}\hspace*{1em}}{\qed\bigskip}

\makeatletter
\def\fnum@figure{{\bf Figure \thefigure}}
\def\fnum@table{{\bf Table \thetable}}
\long\def\@mycaption#1[#2]#3{\addcontentsline{\csname
  ext@#1\endcsname}{#1}{\protect\numberline{\csname
  the#1\endcsname}{\ignorespaces #2}}\par
  \begingroup
    \@parboxrestore
    \small
    \@makecaption{\csname fnum@#1\endcsname}{\ignorespaces #3}\par
  \endgroup}
\def\mycaption{\refstepcounter\@captype \@dblarg{\@mycaption\@captype}}
\makeatother

\newcommand{\mathify}[1]{\ifmmode{#1}\else\mbox{$#1$}\fi}
\newcommand{\bigO}O

\newcommand{\remove}[1]{}
\newcommand{\ignore}[1]{}

\usepackage{amsthm}
\usepackage{thmtools,thm-restate}

\makeatletter
\renewcommand{\paragraph}{%
  \@startsection{paragraph}{4}%
  {\z@}{1.75ex \@plus 1ex \@minus .2ex}{-1em}%
  {\normalfont\normalsize\bfseries}%
}
\makeatother

\newcounter{t}

\declaretheoremstyle[bodyfont=\it,qed=\qedsymbol,headpunct=.\vphantom{$p_{p_{p_p}}$},postheadspace=\newline,headformat=\NAME\  \NUMBER\,\NOTE]{noproofstyle} 

\newtheoremstyle{break}%
{}{}%
{\itshape}{}%
{\bfseries}{.\vphantom{$p_{p_{p_p}}$}}%
{\newline}
{\thmname{#1}\thmnumber{ #2}\thmnote{\ \,\textmd{(#3)}}}

\theoremstyle{break}

\declaretheorem[name=Observation,numbered=no]{observation*}

\declaretheorem[numberlike=t]{fact}

\declaretheorem[numberlike=t]{theorem}

\declaretheorem[name=Theorem,numbered=no]{theorem*}

\declaretheorem[numberlike=t]{lemma}
\declaretheorem[name=Lemma,numbered=no]{lemma*}

\declaretheorem[numberlike=t]{corollary}
\declaretheorem[name=Corollary,numbered=no]{corollary*}

\declaretheorem[name=Parameter,numbered=no]{parameter*}

\declaretheorem[numberlike=t]{proposition}
\declaretheorem[name=Proposition,numbered=no]{proposition*}

\declaretheorem[name=Claim,numbered=no]{claim*}

\declaretheorem[name=Conjecture,numbered=no]{conjecture*}

\declaretheorem[name=Question,numbered=no]{question*}

\declaretheoremstyle[bodyfont=\it,headpunct=.\vphantom{$p_{p_{p_p}}$},postheadspace=\newline,headformat=\NAME\  \NUMBER\,\NOTE]{defstyle}

\declaretheorem[unnumbered,name=Example,style=defstyle]{example*}

\declaretheorem[unnumbered,name=Notation=defstyle]{notation*}

\declaretheorem[unnumbered,name=Construction,style=defstyle]{construction*}

\begin{document}
\title{A useful inequality of inverse hyperbolic tangent}
\author{Kunal Marwaha\thanks{University of Chicago. Email: \texttt{kmarw@uchicago.edu}}}
\date{}
\maketitle

\pagenumbering{arabic}
\vspace{-10mm}
\begin{abstract}
    We prove an inequality related to $\arctanh$, resolving a conjecture of Gu and Polyanskiy~\cite{gu2023weak}. 
\end{abstract}
\vspace{10mm}
In this note, we prove the following statement:
\begin{theorem}
\label{thm:main}
    Choose any $t_1, t_2, t_3 \in (-1, 1)$ and $\lambda \in [0,1]$. Then
    \begin{align}
        (t_1 + t_2 + t_3 + t_1 t_2 t_3) \arctanh \frac{\lambda(t_1 + t_2 + t_3 + t_1 t_2 t_3)}{1 + \lambda(t_1 t_2 + t_2 t_3 + t_3 t_1)} \le \lambda (t_1 + t_2 + t_3 + t_1 t_2 t_3) \sum_{i \in [3]} \arctanh(t_i)\,.
    \end{align}
\end{theorem}

We first show how \Cref{thm:main} resolves the conjecture of~\cite{gu2023weak}. Define the functions $f$ and $g$ as
\begin{align*}
    f(t_1, t_2, t_3) &:= \arctanh(t_1) + \arctanh(t_2) + \arctanh(t_3)\,,
    \\
    g(t_1, t_2, t_3) &:= (t_1 + t_2 + t_3 + t_1 t_2 t_3) f(t_1, t_2, t_3)\,.
\end{align*}

\begin{proposition}
\label{prop:fg_decompose}
$\sum_{i \in [3]} t_i \cdot \arctanh(t_i) = 
        \frac{1}{4} \left(  g(t_1, t_2, t_3) 
        +  g(t_1, -t_2, t_3)
        +  g(t_1, t_2, -t_3)
        +  g(t_1, -t_2, -t_3)
        \right)\,.$
\end{proposition}
\begin{proof}
Notice that $x \cdot \arctanh(x)$ is even, but $\mathbb{E}_{a \in \{\pm 1\}}[a \cdot \arctanh(b x )] = \mathbb{E}_{b \in \{\pm 1\}}[a \cdot \arctanh(b x )] = 0$.
\end{proof}
For any $\lambda \in \mathbb{R}$, define the functions $F_\lambda$ and $G_\lambda$ as
\begin{align*}
    F_\lambda(t_1, t_2, t_3) &:= \arctanh \frac{\lambda(t_1 + t_2 + t_3 + t_1 t_2 t_3)}{1 + \lambda(t_1 t_2 + t_2 t_3 + t_3 t_1)}\,, \\
    G_\lambda(t_1, t_2, t_3) &:= (t_1 + t_2 + t_3 + t_1 t_2 t_3) F_\lambda(t_1, t_2, t_3)\,.
\end{align*}

\begin{corollary}[{\cite[Conjecture 9]{gu2023weak}}]
    Choose any $\lambda, \theta_1, \theta_2, \theta_3 \in [0,1]$. Then the following inequality\footnote{\cite{gu2023weak} notes that the inequality is ``numerically verified''.} holds:
\begin{align}
\frac{1}{4} \left( 
G_\lambda(\theta_1, \theta_2, \theta_3)
+
G_\lambda(\theta_1, -\theta_2, \theta_3)
+
G_\lambda(\theta_1, \theta_2, -\theta_3)
+
G_\lambda(\theta_1, -\theta_2, -\theta_3)
\right)
\le \lambda \sum_{i \in [3]} \theta_i \cdot \arctanh(\theta_i)\,.
\end{align}
\end{corollary}
\begin{proof}
    If any $\theta_i = 1$, the right-hand side is $+\infty$. Otherwise, by \Cref{thm:main}, $G_\lambda (\theta_1, \epsilon_2 \theta_2, \epsilon_3 \theta_3) \le \lambda g(\theta_1, \epsilon_2 \theta_2, \epsilon_3 \theta_3)$ for all $\epsilon_2, \epsilon_3 \in \{-1,1\}$; averaging both sides over $\epsilon_2, \epsilon_3$ and applying \Cref{prop:fg_decompose} implies the result.
\end{proof}

\section{Proof of the main theorem}
We first rewrite $f$ and $F_\lambda$ in terms of logarithms.
\begin{fact}  
    $\arctanh(x) := \frac{1}{2} \ln \frac{1 + x}{1-x}$. When $x = \frac{a}{b}$, then $\arctanh(x) = \frac{1}{2} \ln \frac{b+a}{b-a}$.
\end{fact}

\begin{fact}
    $\prod_{i \in [3]} (1 + \theta_i) + \prod_{i \in [3]} (1 - \theta_i) = 2(1 + \theta_1 \theta_2 + \theta_2 \theta_3 + \theta_3 \theta_1)$.
\end{fact}
\begin{fact}
    $\prod_{i \in [3]} (1 + \theta_i) - \prod_{i \in [3]} (1 - \theta_i) = 2(\theta_1 + \theta_2 + \theta_3 + \theta_1 \theta_2 \theta_3)$.
\end{fact}
\begin{corollary}
The following equalities hold:
\begin{align}
f(t_1, t_2, t_3) &= \frac{1}{2} \ln \frac{\prod_{i \in [3]} (1 + t_i)}{\prod_{i \in [3]} (1 - t_i)}\,,\\
   F_\lambda(t_1, t_2, t_3) &= \arctanh
\frac{\frac{\lambda}{2} \left( \prod_{i \in [3]} (1 + t_i) - \prod_{i \in [3]} (1 - t_i) \right)}
{(1-\lambda) + \frac{\lambda}{2} \left( \prod_{i \in [3]} (1 + t_i) + \prod_{i \in [3]} (1 - t_i) \right)} = \frac{1}{2} \ln \frac{(1-\lambda) + \lambda \prod_{i \in [3]} (1 + t_i) }{(1-\lambda) + \lambda \prod_{i \in [3]} (1 - t_i) }\,.
\end{align}
\end{corollary}

The heart of the proof relies on the following lemma:
\begin{lemma}
    \label{lemma:proof_heart}
    Let $c, d \in \mathbb{R}_{>0}$ such that $0 < c - d < \ln \frac{c}{d}$. Then $\ln \frac{(1-\lambda) + \lambda c}{(1 - \lambda) + \lambda d} \le \lambda \ln \frac{c}{d}$ for all $\lambda \in [0,1]$.
\end{lemma}
\begin{proof}
    Let $h_x(\lambda) = \ln (1 + \lambda(x-1)) - \lambda \ln x$. Suppose $0 < c - d < \ln \frac{c}{d}$; our goal is to prove $h_c(\lambda) \le h_d(\lambda)$ for all $\lambda \in [0,1]$. Note that $h_x(0) = h_x(1) = 0$ for all $x \in \mathbb{R}_{>0}$, and $h_1(\lambda) = 0$ for all $\lambda \in [0,1]$.
    
    When $x \ne 1$, $h'_x(\lambda) = \frac{x-1}{1 + \lambda (x-1)} - \ln x$. This has only one root at $\lambda^*_x = \frac{x-1-\ln x}{(x-1)\ln x}$. 
    Note that $h'_x(0) = x - 1 - \ln x > 0$ for all $x \ne 1$; thus, $h_x(\lambda)$ is increasing for $\lambda \in [0, \lambda_x^*)$ and decreasing for $\lambda \in (\lambda_x^*, 1]$.

    We first consider when one of $c,d$ equals $1$.
    Suppose $c = 1$; then $h_d(\lambda) \ge 0 = h_c(\lambda)$ for all $\lambda \in [0,1]$. When $d = 1$, the premise $c-1 < \ln c$ is never satisfied.

    Now consider any other $c,d \in \mathbb{R}_{>0}$ such that $0 < c - d < \ln \frac{c}{d}$. Here,
    \begin{align}
        h'_c(\lambda) - h'_d(\lambda) = 
        \frac{ \frac{1}{d-1} - \frac{1}{c-1} }
        {(\frac{1}{c-1} + \lambda) (\frac{1}{d-1} + \lambda)} - \ln \frac{c}{d} = \frac{c-d}{\left((c-1)\lambda + 1\right)\left((d-1)\lambda + 1\right)} - \ln \frac{c}{d}\,.
    \end{align}
    The roots of $h'_c(\lambda) - h'_d(\lambda)$ are the roots of the quadratic $(c-1)(d-1)\lambda^2 + (c+d-2)\lambda + 1 - \frac{c-d}{\ln c - \ln d} = 0$.
    Let the roots be $r_1, r_2$ such that $r_1 \le r_2$; then $r_1 r_2 = \frac{1}{(c-1)(d-1)}(1 - \frac{c-d}{\ln c - \ln d})$. The extremum of the quadratic is at $\lambda' = \frac{2-c-d}{2(c-1)(d-1)} = \frac{1}{2-2c} + \frac{1}{2-2d}$; note that $r_1 \le \lambda' \le r_2$. 

    We observe that at most one of $\{r_1, r_2\}$ is in $[0,1]$. When $c > 1$ and $d < 1$, $r_1 r_2 < 0$. Otherwise, $\lambda' \notin [0,1]$: when $c,d < 1$, $\lambda' > 1$, and when $c,d > 1$,  $\lambda' < 0$. So $h'_c(\lambda) - h'_d(\lambda)$ has at most one root for $\lambda \in [0,1]$.

    Note that $h'_c(0) - h'_d(0) = c - d - \ln \frac{c}{d} < 0$ by the premise.
    This implies $h_c(\lambda) - h_d(\lambda) \le 0$ for $\lambda \in [0,1]$, since it equals $0$ when $\lambda = 0$ and $\lambda = 1$, is decreasing at $\lambda = 0$, and has at most one local extremum for $\lambda \in [0,1]$.
\end{proof}
\begin{corollary}
\label{cor:useful}
    Let $\lambda \in [0,1]$ and $t_1, t_2, t_3 \in (-1,1)$ such that $0 < (t_1 + t_2 + t_3 + t_1t_2 t_3) <  f(t_1, t_2, t_3)$. Then
     \begin{align}
     F_\lambda(t_1, t_2, t_3) 
     = \frac{1}{2} \ln \frac{(1-\lambda) + \lambda \prod_{i \in [3]} (1 + t_i) }{(1-\lambda) + \lambda \prod_{i \in [3]} (1 - t_i) }
     \le
     \frac{\lambda }{2} \ln \frac{\prod_{i \in [3]} (1 + t_i)}{\prod_{i \in [3]} (1 - t_i)} 
     = \lambda f(t_1, t_2, t_3) \,.
    \end{align}
\end{corollary}
\begin{proof}
    Let $c := \prod_{i \in [3]} (1 + t_i)$ and  $d := \prod_{i \in [3]} (1 - t_i)$. Then $c - d = 2(t_1 + t_2 + t_3 + t_1t_2 t_3)$ and $\ln \frac{c}{d} = \sum_{i \in [3]} \ln \frac{1 + t_i}{1 - t_i} = 2f(t_1, t_2, t_3)$. The result follows from \Cref{lemma:proof_heart}.
\end{proof}

In our application of \Cref{lemma:proof_heart}, we show that $0 < c - d$ implies $c - d < \ln \frac{c}{d}$:

\begin{lemma}
\label{lemma:conditionistrivial}
Let $t_1, t_2, t_3 \in (-1,1)$ such that $0 < (t_1 + t_2 + t_3 + t_1t_2 t_3)$. Then $(t_1 + t_2 + t_3 + t_1t_2 t_3) <  f(t_1, t_2, t_3)$.
\end{lemma}
\begin{proof}
When $t > 0$, note that by Taylor expansion, $\arctanh(t) > t + \frac{t^3}{3} > t$. So when two variables are zero (let's say $t_1 = t_2 = 0$), $t_3 > 0$ implies $t_3 < \arctanh(t_3) = f(0, 0, t_3)$.

Suppose exactly one variable is zero (let's say $t_1 = 0$). Fix $t_2$. The equation $s(t_3) = t_2 + t_3 - f(0, t_2, t_3)$ has a root at $t_3 = -t_2$. Its derivative is $s'(t_3) = 1 - \frac{1}{1-t_3^2}$, which is negative for all $t_3 \ne 0$ (and zero at $t_3 = 0$). So $t_3 > -t_2$ (i.e. $t_2 + t_3 > 0$) implies $s(t_3) < s(-t_2) = 0$.

Suppose all variables are nonzero. Fix $t_1, t_2$. Consider $s(t_3) = t_1 + t_2 + t_3 + t_1t_2t_3 - f(t_1, t_2, t_3)$. There is a root at $t_3^* := -\frac{t_1 + t_2}{1 + t_1 t_2}$, since at $t_3 = t_3^*$, $\prod_{i \in [3]} (1 + t_i) = \prod_{i \in [3]} (1 - t_i)$. By the previous paragraph, $s(0) < 0$. The derivative is $s'(t_3) = 1 + t_1 t_2 - \frac{1}{1 - t_3^2}$, which has roots at $r_{\pm} = \pm \sqrt{\frac{t_1 t_2}{1 + t_1 t_2}}$. We split the analysis into cases:
\begin{itemize}
\item If $t_1 t_2 < 0$, then $s'(0) = t_1 t_2 < 0$. Since $r_{\pm}$ are not real, $s$ is strictly decreasing. So $t_3 > t_3^*$ (i.e. $t_1 + t_2 + t_3 + t_1 t_2 t_3 > 0$) implies $s(t_3) < 0$.

    \item If $t_1$ and $t_2$ are both negative, then $t_3^* > 0$. Note that $t_3^* > \max(r_+, r_-)$, since
    $(t_1 + t_2)^2 > 2 t_1 t_2 \ge t_1 t_2 (1 + t_1 t_2)$.
    Since $\lim_{t_3 \to 1} s(t_3) = - \infty$, $s(t_3)$ is decreasing when $t_3 > t_3^*$, and so $s(t_3) < 0$ in this region.
    \item If $t_1,t_2$ are both positive, then $t_3^* < 0$. We know $s(t_3) < 0$ for all $t_3^* < t_3 < 0$, since $s(0) < 0$, $s'(0) > 0$, and only one root of $s'$ is negative. When $t_3 > 0$, Taylor expansion gives $s(t_3) < t_1 t_2 t_3 - \frac{1}{3}(t_1^3 + t_2^3 + t_3^3)$. The right-hand side is equal to $-\frac{1}{3}(t_1 + t_2 +t_3)(t_1^2 + t_2^2 + t_3^2 - t_1t_2 - t_2t_3 - t_3t_1)$, which is at most zero.
\end{itemize}
\vspace{-3mm}
\end{proof}
\vspace{-4mm}
\begin{proof}[Proof of \Cref{thm:main}]
    Let $c := \prod_{i \in [3]} (1 + t_i)$ and $d := \prod_{i \in [3]} (1 - t_i)$. If $c = d$, then $G_\lambda(t_1, t_2, t_3)  = g(t_1, t_2, t_3)  = 0$ and the inequality is satisfied. Otherwise, suppose $c < d$. Note that both $G_\lambda$ and $g$ are invariant with respect to flipping the signs of all inputs, so we can consider $G_\lambda(-t_1, -t_2, -t_3)$ and $g(-t_1, -t_2, -t_3)$. But this swaps $c$ and $d$, so we can assume $c > d$.
   Then by \Cref{cor:useful} and \Cref{lemma:conditionistrivial}, $G_\lambda(t_1, t_2, t_3) = \frac{(c-d)}{2}F_\lambda(t_1, t_2, t_3) \le \lambda \frac{(c-d)}{2} f(t_1, t_2, t_3) = \lambda g(t_1, t_2, t_3)$.
\end{proof}
\vspace{-3mm}
\section*{Acknowledgements}
Thanks to the authors of~\cite{gu2023weak} for mentioning this problem in the \href{https://cmsa.fas.harvard.edu/event/gramsia2023/}{GRAMSIA workshop} at Harvard CMSA, and for encouraging this document. Thanks to Neng Huang for comments on a draft of this document.
\vspace{-3mm}
\bibliography{research.bib}
\bibliographystyle{alpha-betta-url}

\end{document}